\documentclass[12pt, reqno]{amsart}

\usepackage[utf8]{inputenc}
\usepackage{amsmath, amsthm, amssymb, amsfonts}
\usepackage{mathtools}
\usepackage{hyperref} 
\usepackage{geometry} 
\usepackage{csquotes}
\usepackage{interval}
\intervalconfig{soft open fences}

\usepackage{enumitem}
\setlist[enumerate]{
    label=(\roman*)
}

\newtheorem{thm}{Theorem}[section]
\newtheorem{lemma}[thm]{Lemma}
\newtheorem{prop}[thm]{Proposition}

\newtheorem{remark}[thm]{Remark}

\theoremstyle{definition}
\newtheorem{question}{Question}[section]

\theoremstyle{definition}
\newtheorem{definition}[thm]{Definition}
\theoremstyle{definition}
\newtheorem{ex}[thm]{Example}

\newcommand{\doi}[1]{DOI: \href{https://doi.org/#1}{#1}}

\makeatletter
\newenvironment{pf}[1][\proofname]{\par
  \pushQED{\qed}%
  \normalfont \topsep-3\p@\relax
  \trivlist
  \item[\hskip\labelsep\itshape
  #1\@addpunct{.}]\ignorespaces
}{%
  \popQED\endtrivlist\@endpefalse
}
\makeatother

\newcommand\Set[2]{\{\,#1\mid#2\,\}}
\newcommand\norm[1]{\left\lVert#1\right\rVert}
\newcommand\abs[1]{\left|#1\right|}

\newcommand\supp{\mathrm{supp}}

\newcommand\N{\mathbb{N}}
\newcommand\Z{\mathbb{Z}}
\newcommand\R{\mathbb{R}}

\newcommand\en{(e_n)_{n\in\N}}
\newcommand\lambdan{(\lambda_n)}
\newcommand\lambdaen{(\lambda_n e_n)}
\newcommand\chin{(\chi_n)_{n\in\N}}

\newcommand\restr[2]{{
  \left.\kern-\nulldelimiterspace 
  #1 
  \right|_{#2} 
  }}

\title{On the refined properties of uniformly quasi-greedy bases}

\author[P. Yu]{Peiyang Yu}
\address{SAM, ETH Zurich, Switzerland}
\email{peiyang.yu@sam.math.ethz.ch}

\date{\today}

\begin{document}

\begin{abstract}
In this paper, we study some refined properties of uniformly quasi-greedy bases. In particular, we characterize the stability of uniformly quasi-greediness under scalings with bounded quotient and show that unconditionality is sufficient to guarantee this property. For the ``isometric" case, i.e., when the uniformly quasi-greedy constant is $1$, we prove that 1-uniformly quasi-greediness implies disjointness if the underlying norm is strictly monotone. We also introduce a stronger notion of quasi-greediness by requiring uniform order boundedness of the greedy sums over greedy orderings. The characterization and properties of such bases are derived along the lines of uniformly quasi-greedy bases.
\end{abstract}

\maketitle
\tableofcontents

\section{Introduction}
Given a Schauder basis $\en$ of a Banach space $X$, the greedy algorithm is a non-linear approximation scheme that approximates any vector $x\in X$ by its coordinates with the largest magnitude. Such a scheme is widely used in signal processing, where a low rank approximation is often sought for a given signal in an infinite or high dimensional Banach space. For an orthonormal basis, it is clear that the greedy approximation converges optimally in norm. More generally, we call a basis $\en$ \textbf{greedy} if the greedy approximation is optimal up to a multiplicative constant, and \textbf{quasi-greedy} if the greedy approximation of all $x\in X$ converges in norm to $x$. The characterization and properties of such bases have been studied comprehensively, see~\cite[Ch.~10]{TopicsBanach} for a summary. \par
Besides norm convergence, it is also of practical interest to require point-wise convergence. This notion of convergence nicely fits into the setting of Banach lattices, since dominated almost everywhere convergence corresponds to order convergence in function spaces. For quasi-greedy bases, the requirement of order convergence of greedy approximation leads to the definition of \textbf{uniformly quasi-greedy bases}, which was recently introduced in~\cite{UQGBases}. Most properties of quasi-greedy bases can be generalized to uniformly quasi-greedy bases, in the sense that if we apply the results for the ambient Banach lattice $C(\interval{0}{1})$, the standard results on quasi-greedy bases are recovered. Some properties, however, still remain mysterious. For example, it is not clear whether uniformly quasi-greediness is stable under any scaling that is bounded and bounded away from 0, which, for quasi-greedy bases, is a consequence of unconditionality for constant coefficients. One may also wonder if 1-uniformly quasi-greedy bases have upgraded properties, since 1-quasi-greediness implies 1-suppression-unconditionality~\cite{1QG}. In this paper, we give partial results on these refined properties of uniformly quasi-greedy bases. In particular, we show that unconditional uniformly quasi-greedy bases are stable under any scaling with bounded quotient, and that 1-uniformly quasi-greedy bases are disjoint if the norm is strictly monotone. We also introduce a stronger notion of uniform quasi-greediness, termed \textbf{absolute quasi-greediness}, where the greedy approximation of an element $x\in X$ is uniformly order bounded over all greedy orderings. The characterization and stability under scaling of these bases are developed along the lines of uniformly quasi-greedy bases. \par
The rest of the paper is organized as follows. In Section 2, we recall some necessary terminology in Banach lattices. In Section 3, we begin with results on the stability under scaling of uniformly quasi-greediness, followed by a result on 1-uniformly quasi-greedy bases, and end with the characterization and properties of absolutely quasi-greedy bases.

\section{Preliminaries}
Throughout the paper, all Banach spaces are over $\R$.
\subsection{Notions of convergence}
Let $X$ be a vector lattice and $(x_{\alpha})_{\alpha\in A}\subset X$ be a net. We say that $(x_{\alpha})$ decreases to $0$, denoted as $x_{\alpha}\downarrow 0$, if $(x_{\alpha})$ is decreasing and $\inf \{x_{\alpha}\} =0$. $X$ is called \textbf{Archimedean} if $(n^{-1} e)_{n\in \N}\downarrow 0$ for any positive vector $e\in X$. Let us recall the following notions of convergence in an Archimedean vector lattice
\begin{itemize}
    \item $(x_{\alpha})$ is said to \textbf{converge uniformly} to $x\in X$, denoted as $x_{\alpha}\xrightarrow{u}x$, if there exists $e\in X$ such that
    \begin{equation*}
        \forall \epsilon >0:\exists \alpha_0\in A:\; \alpha\geq \alpha_0 \implies \abs{x_{\alpha}-x}\leq \epsilon e.
    \end{equation*}
    \item $(x_{\alpha})$ is said to \textbf{converge in order} to $x\in X$, denoted as $x_{\alpha}\xrightarrow{o}x$, if there exists a net $(u_{\gamma})$ with a possibly different index set $B$ such that $u_{\gamma}\downarrow 0$ and
    \begin{equation*}
        \forall \gamma \in B: \exists \alpha_0\in A:\; \alpha\geq \alpha_0 \implies \abs{x_{\alpha}-x}\leq u_\gamma.
    \end{equation*}
    \item A sequence $(x_n)$ is said to \textbf{$\sigma$-order converge} to $x$, denoted as $x_n\xrightarrow{\sigma o}x$, if there exists a sequence $(u_n)$ such that $u_n\downarrow 0$ and $\abs{x_n-x}\leq u_n, \forall n\in\N$. 
\end{itemize}
It is easy to see that
\begin{equation*}
    x_n\xrightarrow{u}x \implies x_n\xrightarrow{\sigma o}x \implies x_n\xrightarrow{o}x.
\end{equation*}
It can be shown that $\sigma$-order convergence and order convergence agree on $\sigma$-order complete vector lattices; in general, however, the converse of none of the above implications holds true. In a Banach lattice, uniform convergence also implies norm convergence. Norm convergent sequences are in general not order convergent, but there always exists a uniformly convergent subsequence.
\subsection{Notions of bases}
Recall that a sequence $\en$ in a Banach space $X$ is called a \textbf{(Schauder) basis} of $X$ if, for all $x\in X$, there exists a unique series expansion $x=\sum_{n=1}^{\infty} a_n e_n$ converging in norm. For each $m\in\N$, we can then define the \textbf{$\mathbf{m}$-th basis projection} $P_m:X\rightarrow X$ via $P_m(\sum_{n=1}^{\infty}a_n e_n)\coloneqq \sum_{n=1}^{m}a_n e_n$ and the \textbf{biorthogonal functional} $e^*_m:X\rightarrow \R$ via $e^*_m(\sum_{n=1}^{\infty}a_n e_n) \coloneqq a_m$. It is known that the $P_m$'s are uniformly bounded and the number $K\coloneqq \norm{P_m}$ is called the \textbf{basis constant} of $\en$. A sequence $\en$ is called a \textbf{(Schauder) basic sequence} if it is a Schauder basis of its closed linear span $[\en]$, in which case the $P_m$'s and $e_m^*$'s are defined on $[\en]$. It is a standard fact that a sequence of non-zero vectors $\en$ is basic if and only if there exists $C\geq 1$ such that 
\begin{equation}\label{max ineq:basic}
    \bigvee_{1\leq n\leq m}\norm{\sum_{k=1}^n\alpha_k e_k}\leq C\norm{\sum_{k=1}^m\alpha_k e_k}
\end{equation}
for all $m \in \N$ and all scalars $\alpha_1, \ldots, \alpha_m$. The least such constant $C$ is called the \textbf{basis constant} of $\en$ and denoted as $K_b$. A sequence $\en$ is called \textbf{unconditional} if every permutation of $\en$ is a basic sequence, which is known to be characterized by the existence of a constant $C\geq 1$ such that 
\begin{equation}\label{max ineq:uncond}
    \bigvee_{\epsilon_k=\pm 1}\norm{\sum_{k=1}^m \epsilon_k\alpha_k e_k}\leq C\norm{\sum_{k=1}^m\alpha_k e_k}
\end{equation}
for all $m \in \N$ and all scalars $\alpha_1, \ldots, \alpha_m$. We denote the least such constant $C$ as $K_u$ and refer to it as the \textbf{unconditional constant} of $\en$. By the convexity of the norm, we have the following property of an unconditional sequence $\en$ for arbitrary $m\in \N$ and scalars $(\alpha_k)_{k=1}^m$
\begin{equation}\label{eq:uncond:property}
    \abs{\beta_k}\leq \abs{\alpha_k},\, \forall 1\leq k\leq m
    \implies
    \norm{\sum_{k=1}^m \beta_k e_k} \leq K \norm{\sum_{k=1}^m \alpha_k e_k}.
\end{equation}
\par 
Let us now fix $\en$ to be a semi-normalized Schauder basis of a Banach space $E$ and let $x\in E$. A permutation $\pi$ of $\N$ is called a \textbf{greedy ordering} of $x$ if $(\abs{e^*_{\pi(n)}(x)})$ is non-increasing. The \textbf{$\mathbf{m}$-th greedy sum} of $x$ associated to $\pi$ is defined as $G_{\pi,m}(x) \coloneqq \sum_{n=1}^m e^*_{\pi(n)}(x)e_{\pi(n)}$ and the sequence $(G_{\pi,m}(x))_{m=1}^{\infty}$ is called the \textbf{greedy approximation} of $x$ associated to $\pi$. It is clear that $x$ can have different greedy orderings, and $G_{\pi,m}(x)$ may be different from $G_{\rho,m}(x)$ for a fixed $m\in\N$ and different greedy orderings $\pi$ and $\rho$. We call $G_{\pi,m}(x)$ a \textbf{strictly greedy sum} of $x$ of order $m$ if the sum is independent of the greedy ordering or, equivalently, if $\abs{e^*_{\pi(m)}(x)} > \abs{e^*_{\pi(m+1)}(x)}$ or $e^*_{\pi(m)}(x)=0$. The \textbf{natural greedy ordering} of $x$ is the canonical greedy ordering induced by the basis, namely the permutation $\sigma$ of $\N$ such that if $j<k$ then either $\abs{e^*_{\sigma(j)}(x)}>\abs{e^*_{\sigma(k)}(x)}$ or $\abs{e^*_{\sigma(j)}(x)}=\abs{e^*_{\sigma(k)}(x)}$ and $\sigma(j)<\sigma(k)$. The \textbf{$\mathbf{m}$-th natural greedy sum} of $x$ is then defined as $\mathcal{G}_m(x)\coloneqq G_{\sigma,m}(x)$. We call $\en$ \textbf{quasi-greedy} if $\mathcal{G}_m(x)\xrightarrow{\norm{\cdot}} x$ for all $x\in E$, which is characterized by the following  inequality similar to~\eqref{max ineq:basic}
\begin{equation}\label{max ineq:qg}
    \exists C\geq 1: \forall m\in \N, \, x\in E: \norm{\mathcal{G}_m(x)}\leq C\norm{x}.
\end{equation}
The least such constant $C$ is called the \textbf{quasi-greedy constant} of $\en$ and is denoted as $K_{qg}$. It is known that the definition and characterization of quasi-greediness do not depend on the choice of greedy ordering~\cite[Lemma 10.2.5, Lemma 10.2.6]{TopicsBanach}, in the sense that one can equivalently formulate them if the natural greedy ordering is replaced by the existence of a greedy ordering or the requirement for all greedy orderings. If $E$ is embedded in a Banach lattice $X$, we obtain the definition of \textbf{uniformly quasi-greedy bases} by enforcing uniform convergence of $\mathcal{G}_m(x)$ to $x$. The following theorem is proven in~\cite{UQGBases}.
\begin{thm}\label{Charac:UQG}
  Let $\en$ be a semi-normalized basis of $E\subset X$. TFAE:
  \begin{enumerate}
    \item For all $x\in E$, $\mathcal{G}_n(x)\xrightarrow{u}x$;\label{Charac:UQG:i}
    \item For all $x\in E$, $\mathcal{G}_n(x)\xrightarrow{\sigma o}x$;\label{Charac:UQG:ii}
    \item For all $x\in E$, $\mathcal{G}_n(x)\xrightarrow{o}x$;\label{Charac:UQG:iii}
    \item For all $x\in E$, there exists $u\in X$ such that $\abs{\mathcal{G}_n(x)}\leq u$ for all $n\in\N$;\label{Charac:UQG:iv}
    \item For all $x\in E$, there exists $C\geq 0$ such that $\norm{\bigvee_{n=1}^m\abs{\mathcal{G}_n(x)}}\leq C$ for all $m\in\N$;\label{Charac:UQG:v}
    \item There exists $C\geq 1$ such that
    \begin{equation}\label{max ineq:UQG}
      \norm{\bigvee_{n=1}^m\abs{\mathcal{G}_n(x)}}\leq C\norm{x}
    \end{equation}
    for all $m\in\N$ and $x\in E$.\label{Charac:UQG:vi}
  \end{enumerate}
\end{thm}
The least constant satisfying~\eqref{max ineq:UQG} is referred to as the \textbf{uniformly quasi-greedy constant} of $\en$ and is denoted as $K_{uqg}$. It can be shown that the formulations of~\ref{Charac:UQG:i}--\ref{Charac:UQG:vi} in the above theorem do not depend on the choice of greedy orderings either. For the brevity of notation, we will fix the set-up of a Banach space $E$ embedded in a Banach lattice $X$ and consider bases of $E$. Let us define 
\begin{equation*}
  G_{\pi,m}^{\vee}(x)\coloneqq\bigvee_{n=1}^{m}\abs{G_{\pi,m}(x)}
  =\bigvee_{n=1}^{m}\abs{\sum_{k=1}^n e^*_{\pi(k)}(x)e_{\pi(k)}}, 
  \quad 
  \mathcal{G}_{m}^{\vee}(x)\coloneqq\bigvee_{n=1}^m \abs{\mathcal{G}_n(x)}
\end{equation*}
for $x\in E$ and $\pi$ a greedy ordering of $x$, and the following sets for any constant $a\geq 1$ 
\begin{alignat*}{2}
  B_{\mu}^a(x)&\coloneqq\Set{n\in\N}{\abs{e^*_n(x)}\in \interval[open left]{\mu}{a\mu}} = \left\{ r^{\mu}_1, r^{\mu}_2, \ldots, r^{\mu}_{l_{\mu}}\right\},
  &&\quad \forall \mu>0,\\
  A_k^a(x)&\coloneqq B_{a^k}^a(x) = \left\{ s^k_1, s^k_2, \ldots, s^k_{t_k}\right\},
  &&\quad\forall k\in\Z,
\end{alignat*}
where $l_{\mu}$ and $t_k$ are the cardinalities of $B_{\mu}^a(x)$ and $A_k^a(x)$, respectively, and the elements are enumerated by the natural order on $\N$. We use tilde to indicate the notation for the scaled basis $\lambdaen$ for a sequence of real numbers $\lambdan$, e.g. $\tilde{G}_{\pi,m}(x)$, $\tilde{B}_{\mu}^a(x)$, $\tilde{A}_k^a(x)$. For sets $A$, $B$ and $C$, we write $A\sqcup B\sqcup C$ for their union and to mean that they are pairwise disjoint.
\section{Main results}
\subsection{Stability under scaling of uniformly quasi-greediness}
For a quasi-greedy basis $\en$, it is known that for any sequence of real numbers $\lambdan$ with $0<\inf_{n\in\mathbb{N}}\abs{\lambda_n}\leq\sup_{n\in\mathbb{N}}\abs{\lambda_n}<\infty$ the scaled basis $\lambdaen$ is still quasi-greedy. For the generalization of this property to uniformly quasi-greediness, we first note that we can WLOG restrict to the case where $1\leq\inf_{n\in\mathbb{N}}\abs{\lambda_n}\leq\sup_{n\in\mathbb{N}}\abs{\lambda_n}\leq a$ for some $a>1$, since any scaling with bounded quotient can be obtained by a combination of finitely many such scalings and a constant scaling. In the case of quasi-greediness, the proof of the stability under scaling~\cite[Prop. 10.2.16]{TopicsBanach} relies on the following facts.
\begin{itemize}
    \item For any $x\in E$, $\mu>0$, it holds that
    \begin{equation}\label{Scaling:QG:1}
        \begin{aligned}
            \Set{n\in \N}{\abs{e^*_n(x)}\geq \sup_{n\in\mathbb{N}}\abs{\lambda_n}\mu}
            &\subset 
            \Set{n\in \N}{\frac{\abs{e^*_n(x)}}{\abs{\lambda_n}}\geq \mu}\\
            &\subset
            \Set{n\in \N}{\abs{e^*_n(x)}\geq \mu}.
        \end{aligned}
    \end{equation}
    \item For any finite set $A\subset \N$ and scalars $(a_n)_{n\in A}$, we have by~\cite[Cor.~10.2.13]{TopicsBanach}
    \begin{equation}\label{Scaling:QG:2}
        \norm{\sum_{n\in B}a_n e_n}
        \leq 
        16 K_{qg}^4 \frac{\max (\abs{a_n})_{n\in A}}{\min (\abs{a_n})_{n\in A}}\norm{\sum_{n\in A}a_n e_n}.
    \end{equation}
\end{itemize}
By~\eqref{Scaling:QG:1}, we can decompose greedy sums of $\lambdaen$ in the following way
\begin{equation*}
    \sum_{\abs{e_n^*(x)}\geq \abs{\lambda_n}\mu}
    e_n^*(x) e_n
    = \sum_{\abs{e_n^*(x)}\geq \sup_{n\in\N}\abs{\lambda_n}\mu} e_n^*(x)e_n
    + \sum_{\abs{\lambda_n}\mu \leq \abs{e_n^*(x)}<\sup_{n\in\N}\abs{\lambda_n}\mu} e_n^*(x)e_n.
\end{equation*}
The first term on the right-hand side is controlled by~\eqref{max ineq:qg}, and the second by~\eqref{Scaling:QG:2} and~\eqref{max ineq:qg}. If we try to establish uniformly quasi-greediness of the scaled basis by the same decomposition, we are faced with the stronger condition that the second term on the right needs to be uniformly order bounded over all $\mu>0$. Since the scaling has a bounded quotient, it suffices to consider ``dyadic'' segments of the basis expansion of $x$. By choosing the scaling freely, we arrive at the following characterization of the stability under scaling of uniformly quasi-greediness, which extends a result from the author's Bachelor thesis~\cite[Prop.~4.15]{BibasicUQGThesis}.\par
\begin{prop}\label{Scaling:Equiv:1}
  Let $\en$ be a uniformly quasi-greedy basis and $a>1$. TFAE
  \begin{enumerate}
    \item For all $\lambdan$ with $1\leq\abs{\lambda_n}\leq a$, $\lambdaen$ is uniformly quasi-greedy;\label{Scaling:Equiv:1:i}
    \item For all $x\in E$ and any permutation $\pi:\N \rightarrow\N$ such that $\pi(A_k^a(x))=A_k^a(x)$ for all $k\in \Z$, there exists $u\in X$ such that 
    \begin{equation*}
      \forall k\in\Z:\;\bigvee_{n=1}^{t_k}\abs{\sum_{i=1}^{n} e^*_{\pi(s_i^k)}(x)e_{\pi(s_i^k)}}\leq u.
    \end{equation*}\label{Scaling:Equiv:1:ii}
    \item For all $x\in E$ and any permutation $\pi:\N \rightarrow\N$ such that $\pi(A_k^a(x))=A_k^a(x)$ for all $k\in \Z$, there exists $C>0$ such that 
    \begin{equation*}
      \forall N \in \N:\; \norm{\bigvee_{k=-N}^{N}\bigvee_{n=1}^{t_k}\abs{\sum_{i=1}^{n} e^*_{\pi(s_i^k)}(x)e_{\pi(s_i^k)}}}\leq C.
    \end{equation*}\label{Scaling:Equiv:1:iii}
  \end{enumerate}
\end{prop}
\begin{proof}
  Assume that~\ref{Scaling:Equiv:1:i} holds. Let $x\in E$ and $\pi$ be a permutation such that $\pi(A_k^a(x))=A_k^a(x)$ for all $k\in \Z$. Since $(A_k^a(x))_{k\in\Z}$ is a partition of $\supp(x)$, we can define $\lambdan$ via
  \begin{equation*}
    \lambda_n \coloneqq
    \begin{dcases}
        \frac{\abs{e^*_n(x)}}{\min_{n\in A_k^a(x)}(\abs{e^*_n(x)})} 
        & \text{ if } n\in A_k^a,\\
        1 & \text{ if } n\notin \supp(x).
    \end{dcases}
  \end{equation*}
  This implies that $\frac{\abs{e^*_n(x)}}{\lambda_n}=\min_{n\in A_k^a(x)}(\abs{e^*_n(x)})$ for all $n \in A_k^a(x)$. In particular, $\frac{\abs{e^*_n(x)}}{\lambda_n}$ is constant on $A_k^a(x)$ and $\pi$ is a greedy ordering of $x$ with respect to $\lambdaen$ on $A_k^a(x)$. By rearranging $\pi$, we get a global greedy ordering $\pi^{\prime}$ of $x$ with respect to $\lambdaen$ that keeps the ordering of $\pi$ on $A_k^a(x)$ for all $k\in\Z$. By assumption, $\lambdaen$ is uniformly quasi-greedy, so by Thm.~\ref{Charac:UQG}~\ref{Charac:UQG:iv}, there exists a $u\in X$ such that $\abs{\tilde{G}_{\pi^{\prime},m}(x)}\leq u$ for all $m \in\N$. We conclude by observing that for all $k\in\Z$ and $1\leq n\leq t_k$, $\sum_{i=1}^{n} e^*_{\pi(s_i^k)}(x)e_{\pi(s_i^k)}=\sum_{i=1}^{n} \frac{e^*_{\pi(s_i^k)}(x)}{\lambda_{\pi(s_i^k)}}\lambda_{\pi(s_i^k)}e_{\pi(s_i^k)}$ is the difference of $\tilde{G}_{\pi^{\prime},m}(x)$ and $\tilde{G}_{\pi^{\prime},l}(x)$ for some $m,l\in\N$, which implies that
  \begin{equation*}
    \forall k\in\Z:\forall 1\leq n\leq t_k:\; 
    \abs{\sum_{i=1}^{n} e^*_{\pi(s_i^k)}(x)e_{\pi(s_i^k)}}
    \leq \abs{\tilde{G}_{\pi^{\prime},m}(x)} + \abs{\tilde{G}_{\pi^{\prime},l}(x)}\leq 2u.
  \end{equation*}
  \ref{Scaling:Equiv:1:ii}$\Rightarrow$\ref{Scaling:Equiv:1:iii} is clear. For~\ref{Scaling:Equiv:1:iii}$\Rightarrow$\ref{Scaling:Equiv:1:i}, let $\lambdan$ with $1\leq\inf_{n\in\mathbb{N}}\abs{\lambda_n}\leq\sup_{n\in\mathbb{N}}\abs{\lambda_n}\leq a$. To show that $\lambdaen$ is uniformly quasi-greedy, we resort to Thm.~\ref{Charac:UQG}~\ref{Charac:UQG:v} and show that for all $x\in E$, there is a greedy ordering $\sigma$ of $x$ with respect to $\lambdaen$ and $C>0$ such that $\norm{\tilde{G}_{\sigma,l}^{\vee}(x)}\leq C$ for all $m\in\N$. Let $x\in E$ and $\sigma$ be a greedy ordering of $x$ with respect to $\lambdaen$. For $m \in\N$, let $k\in \Z$ be such that $\abs{\frac{e^*_{\sigma(m)}(x)}{\lambda_{\sigma(m)}}}\in \interval[open left]{a^k}{a^{k+1}}$. We have then
  \begin{align*}
    \Set{n\in\N}{\abs{e^*_n(x)}>a^{k+2}} 
    &\subset \Set{n\in\N}{\abs{\frac{e^*_n(x)}{\lambda_n}}>\abs{\frac{e^*_{\sigma(m)}(x)}{\lambda_{\sigma(m)}}}} \\
    &\subset \sigma(\{1,\ldots,m\}) \\
    &\subset \Set{n\in\N}{\abs{\frac{e^*_n(x)}{\lambda_n}}\geq\abs{\frac{e^*_{\sigma(m)}(x)}{\lambda_{\sigma(m)}}}}
    \subset \Set{n\in\N}{\abs{e^*_n(x)}\geq a^k}
  \end{align*}
  and hence, denoting $\sigma(\{1,\ldots,m\})$ as $S_m$,
  \begin{gather}
    S_m = \Set{n\in\N}{\abs{e^*_n(x)}>a^{k+2}} \sqcup (S_m\cap A^a_{k+1}(x)) \sqcup (S_m\cap A^a_k(x)),\nonumber\\
    \abs{\tilde{G}_{\sigma,m}(x)} 
    \leq \abs{\sum_{\abs{e^*_n(x)}>a^{k+2}} e^*_n(x)e_n} + \abs{\sum_{S_m\cap A^a_{k+1}(x)} e^*_n(x)e_n} + \abs{\sum_{S_m\cap A^a_{k}(x)} e^*_n(x)e_n}.\label{Ineq:1}
  \end{gather}
  The first term in~\eqref{Ineq:1} is a strictly greedy sum with respect to $\en$. Hence, the supremum over finitely many terms in the form of the first term is norm bounded by some $C_1$ by the assumption that $\en$ is uniformly quasi-greedy and Thm.~\ref{Charac:UQG}~\ref{Charac:UQG:v}. For the second and third term, let us define a permutation $\pi$ such that $\pi(A_k^a(x))=A_k^a(x)$ for all $k\in\Z$ and the order of $\pi$ on $A_k^a(x)$ is given by $\sigma$, namely 
  \begin{equation*}
    \forall s^k_i, s^k_j \in A^a_k(x): \sigma^{-1}(s^k_i) > \sigma^{-1}(s^k_j) \implies \pi^{-1}(s^k_i) > \pi^{-1}(s^k_j).
  \end{equation*}
  For all $k\in\Z$, we have
  \begin{equation*}
    \abs{\sum_{S_m\cap A^a_{k}(x)} e^*_n(x)e_n} = 
    \abs{\sum_{i=1}^{\abs{S_m\cap A^a_k(x)}} e^*_{\pi(s_i^k)}(x)e_{\pi(s_i^k)}}.
  \end{equation*}
  Note that the sums on the right-hand side grow in the order of $\pi$ as $m$ increases. Taking supremum over $1\leq m\leq l$, these are norm bounded by some $C$ by assumption, where $C$ is independent of $l$. Hence, taking supremum and then taking norm on both sides of~\eqref{Ineq:1}, we obtain $\norm{\tilde{G}^{\vee}_{\sigma,m}(x)}\leq C_1+2C$ for all $m\in\N$ by the above estimates. 
\end{proof}
Using the above characterization, we obtain the following sufficient condition for the stability under scaling of uniformly quasi-greediness.
\begin{prop}\label{Scaling:Uncond:1}
  Let $\en$ be an unconditional uniformly quasi-greedy basis. Then $\lambdaen$ is uniformly quasi-greedy for all $\lambdan$ with $0<\inf_{n\in\mathbb{N}}\abs{\lambda_n}\leq\sup_{n\in\mathbb{N}}\abs{\lambda_n}<\infty$.
\end{prop}
\begin{proof}
  Let $a>1$. We establish~\ref{Scaling:Equiv:1:iii} in Prop.~\ref{Scaling:Equiv:1} by exploiting unconditionality. Let $x\in E$ and $\pi$ a permutation such that $\pi(A^a_k(x)) = A^a_k(x)$ for all $k\in\Z$. Similar to the proof of Prop.~\ref{Scaling:Equiv:1}, we define $\lambdan$ such that $\frac{\abs{e^*_n(x)}}{\lambda_n}=\min_{n\in A_k^a(x)}(\abs{e^*_n(x)})$ for all $n \in A_k^a(x)$, which implies that $\pi$ is a greedy ordering of $x$ with respect to $\lambdaen$ on $A_k^a(x)$. By rearranging $\pi$, we get a global greedy ordering $\pi^{\prime}$ of $x$ with respect to $\lambdaen$ that keeps the ordering of $\pi$ on $A_k^a(x)$ for all $k\in\Z$.\par 
  Since $\lambdan$ is bounded from below and $\en$ is unconditional, $\bar{x}\coloneqq\sum_{n=1}^{\infty}\frac{e^*_n(x)}{\lambda_n} e_n$ is well-defined and we observe that $\pi'$ is also a greedy ordering of $\bar{x}$ with respect to $\en$. Using the assumption that $\en$ is uniformly quasi-greedy, it holds that $\norm{G_{\pi^{\prime},m}^{\vee}(\bar{x})} \leq K_{uqg}\norm{\bar{x}}\leq K_{uqg} K_u\norm{x}$ for all $m\in\N$, where $K_{uqg}$ and $K_u$ denote the uniformly quasi-greedy and unconditional constant of $\en$, respectively. This implies that
  \begin{equation*}
      \forall N \in \N:\; \norm{\bigvee_{k=-N}^{N}\bigvee_{n=1}^{t_k}\abs{\sum_{i=1}^{n} \frac{e^*_{\pi(s_i^k)}(x)}{\lambda_{\pi(s_i^k)}}e_{\pi(s_i^k)}}}\leq 2K_{uqg}K_u\norm{x}
  \end{equation*}
  since the sums in the supremum are differences of greedy sums of $\bar{x}$ associated to $\pi^{\prime}$. Now, by changing the signs of $\lambdan$ on $\supp(x)$ and going through the same argument, we obtain that 
  \begin{equation*}
      \forall N \in \N:\; \norm{\bigvee_{k=-N}^{N}\bigvee_{n=1}^{t_k}\abs{\sum_{i=1}^{n} \frac{\epsilon_{\pi(s_i^k)} e^*_{\pi(s_i^k)}(x)}{\lambda_{\pi(s_i^k)}}e_{\pi(s_i^k)}}}\leq 2K_{uqg}K_u\norm{x}
  \end{equation*}
  for all sign sequences $(\epsilon_n)_{n\in\supp(x)}$. Let $N\in\N$ and denote $K_N\coloneqq \bigcup_{k=-N}^N A^a_k$. Since $\abs{\lambda_n}/a\leq 1$, by~\cite[Thm. 3.13]{BasisPrimer}, there exists $c_l\geq 0$ and $\epsilon_{l}^n$ for $n\in K_N$ and $1\leq l\leq \abs{K_N}+1$ such that 
  \begin{equation*}
    \sum_{l=1}^{\abs{K_N}+1} c_l = 1,\quad 
    \sum_{l=1}^{\abs{K_N}+1} c_l\epsilon_{l}^n = \frac{\lambda_n}{a},\quad \forall n\in K_N.
  \end{equation*}
  By convexity, we obtain that 
  \begin{align*}
    \norm{\bigvee_{k=-N}^{N}\bigvee_{n=1}^{t_k}\abs{\sum_{i=1}^{n} e^*_{\pi(s_i^k)}(x)e_{\pi(s_i^k)}}}
    &= \norm{\bigvee_{k=-N}^{N}\bigvee_{n=1}^{t_k}\abs{\sum_{i=1}^{n} \sum_{l=1}^{\abs{K_m}+1} a c_l\epsilon_{l}^{\pi(s_i^k)}\frac{e^*_{\pi(s_i^k)}(x)}{\lambda_{\pi(s_i^k)}}e_{\pi(s_i^k)}}}\\
    &\leq \norm{\bigvee_{k=-N}^{N}\bigvee_{n=1}^{t_k} \sum_{l=1}^{\abs{K_m}+1} a c_l \abs{\sum_{i=1}^{n} \epsilon_{l}^{\pi(s_i^k)}\frac{e^*_{\pi(s_i^k)}(x)}{\lambda_{\pi(s_i^k)}}e_{\pi(s_i^k)}}}\\
    &\leq \norm{\sum_{l=1}^{\abs{K_m}+1} a c_l \bigvee_{k=-N}^{N}\bigvee_{n=1}^{t_k} \abs{\sum_{i=1}^{n} \epsilon_{l}^{\pi(s_i^k)}\frac{e^*_{\pi(s_i^k)}(x)}{\lambda_{\pi(s_i^k)}}e_{\pi(s_i^k)}}}\\
    &\leq \sum_{l=1}^{\abs{K_m}+1} a c_l \norm{ \bigvee_{k=-N}^{N}\bigvee_{n=1}^{t_k} \abs{\sum_{i=1}^{n} \epsilon_{l}^{\pi(s_i^k)}\frac{e^*_{\pi(s_i^k)}(x)}{\lambda_{\pi(s_i^k)}}e_{\pi(s_i^k)}}}\\
    &\leq 2a K_{uqg}K_u\norm{x}.\qedhere
  \end{align*}
\end{proof}
\begin{ex}
Tao~\cite{WaveletTao} showed that the strictly greedy sums of compactly supported wavelets in $L^p(\interval{0}{1})$, $1<p<\infty$ are uniformly bounded by the Hardy-Littlewood maximal function. Hence, by Thm.~\ref{Charac:UQG}~\ref{Charac:UQG:iv} and a perturbation argument, these bases are always uniformly quasi-greedy. In particular, the (normalized) Haar system forms an unconditional and uniformly quasi-greedy basis of $L^p(\interval{0}{1})$, $1<p<\infty$. By Prop.~\ref{Scaling:Uncond:1}, the scaled Haar system is still uniformly quasi-greedy for all scalings with bounded quotient.
\end{ex}
It is clear that unconditionality is not necessary for the stability under scaling of uniformly quasi-greediness, since any conditional quasi-greedy basis (see e.g.~\cite{L1AQG}) can be embedded into $C(\interval{0}{1}$, where the notions of quasi-greedy bases and uniformly quasi-greedy bases coincide. We now give a necessary condition for the stability under scaling of uniformly quasi-greediness, which was also presented in the author's Bachelor thesis~\cite[Prop.~4.14]{BibasicUQGThesis}.
\begin{prop}\label{Scaling:Property:1}
  If $\lambdaen$ is uniformly quasi-greedy for all $\lambdan$ with $1\leq\abs{\lambda_n}\leq a$, then the supremum of the uniformly quasi-greedy constants over all such $\lambdaen$ is finite. 
\end{prop}
\begin{proof}
  We prove by contradiction. Assume that the supremum is infinite. We claim then that the supremum of the uniformly quasi-greedy constants over all such $(\lambda_n e_n)_{n\geq 2}$ is also infinite. Suppose not. Then there exists $M\geq 1$ such that for all $x\in E$ with finite support and $e^*_1(x)=0$, for all scalings $\lambdan$ with $1\leq\abs{\lambda_n}\leq a$ and for all greedy orderings $\pi_{\lambda}$ of $x$ with respect to $(\lambda_n e_n)_{n\geq 2}$, we have
  \begin{equation*}
    \norm{\tilde{G}_{\pi_{\lambda},\abs{\supp(x)}}^{\vee}(x)} \leq M\norm{x}.
  \end{equation*}
  For any $x\in E$ with finite support and $e^*_1(x)\neq 0$, a scaling $\lambdan$ with $1\leq\abs{\lambda_n}\leq a$, and a greedy ordering $\pi_{\lambda}$ of $x$ with respect to $\lambdaen$, we have then
  \begin{equation*}
    \tilde{G}_{\pi_{\lambda},\abs{\supp(x)}}^{\vee}(x) 
    \leq \abs{e^*_1(x)e_1} + \tilde{G}_{\pi_{\lambda},\abs{\supp(x)}-1}^{\vee}(x-e^*_1(x)e_1).
  \end{equation*} 
  Taking norm on both sides yields
  \begin{align*}
    \norm{\tilde{G}_{\pi_{\lambda},\abs{\supp(x)}}^{\vee}(x)} 
    &\leq \norm{e^*_1(x)e_1} + \norm{\tilde{G}_{\pi_{\lambda},\abs{\supp(x)}-1}^{\vee}(x-e^*_1(x)e_1)}\\
    &\leq \norm{e^*_1(x)e_1} + M\norm{x-e^*_1(x)e_1}\\
    &\leq (1+M)\norm{e^*_1(x)e_1} + M\norm{x}\\
    &\leq ((1+M)K_b+M)\norm{x},
  \end{align*}
  where $K_b$ denotes the basis constant of $\en$. This contradicts the assumption and therefore proves the claim. Proceeding inductively, we obtain that the supremum of the uniformly quasi-greedy constants over all $(\lambda_n e_n)_{n\geq N}$ with $1\leq\abs{\lambda_n}\leq a$ is infinite for all $N\in\N$.\par
  By assumption, there exists an $x_1\in E$ with finite support, a scaling with $1\leq\abs{\lambda^1_n}\leq a$ and a greedy ordering $\pi_1$ of $x$ with respect to $(\lambda_n^1 e_n)_{n\in\N}$ such that 
  \begin{equation*}
    \norm{\tilde{G}^{\vee}_{\pi_1,\supp(x_1)}(x_1)} \geq \norm{x_1}.
  \end{equation*}  
  Let $N_1 \coloneqq \max_{n\in \supp(x)} n$. By the above claim, we can find an $x_2\in E$ with finite support in $\N\setminus\{1,\ldots N_1\}$, a scaling with $1\leq\abs{\lambda^2_n}\leq a$ and a greedy ordering $\pi_2$ of $x_2$ such that 
  \begin{equation*}
    \norm{\tilde{G}^{\vee}_{\pi_2,\supp(x_2)}(x_2)} \geq 2\norm{x_2}.
  \end{equation*}
  Proceeding iteratively, we find a sequence of elements $(x_k)_{k\in\N}$ with finite and disjoint support such that for every $k\in\N$ there exists a scaling $(\lambda_n^k)_{n\in\N}$ and a greedy ordering $\pi_k$ of $x_k$ with respect to $(\lambda_n^k e_n)_{n\in\N}$ that satisfies
  \begin{equation*}
    \norm{\tilde{G}^{\vee}_{\pi_k,\supp(x_k)}(x_k)} \geq k\norm{x_k}.
  \end{equation*}
  Since $(x_k)_{k\in\N}$ have disjoint support, we can define a scaling $\lambdan$ that coincides with $(\lambda_n^k)_{n\in\N}$ on $\supp(x_k)$. The resulting scaled basis $\lambdaen$ then fails Thm.~\ref{Charac:UQG}~\ref{Charac:UQG:vi}, a contradiction.
\end{proof}
\begin{question}
    If $X$ is not an AM-space, are conditional uniformly quasi-greedy bases stable under scalings with bounded quotient? 
\end{question}
\subsection{1-uniformly quasi-greedy bases}
If $X$ is an AM-space, uniformly quasi-greediness is equivalent to quasi-greediness, and it has been shown that 1-quasi-greediness is equivalent to 1-suppression-unconditionality~\cite{1QG}. Inspired by this result, one may wonder whether $K_{uqg}=1$ implies that $\en$ is 1-absolute. Let us first recall that the norm on $X$ is said to be strictly monotone if
\begin{equation*}
    \forall x,y\in X: \; \abs{x}<\abs{y}\implies \norm{x}<\norm{y}.
\end{equation*}
The strict monotonicity of the norm allows norm equality to reveal order relations, which makes 1-uniformly quasi-greediness a fairly strong condition.
\begin{prop}\label{1-UQG:monotone norm}
    Assume that the norm on $X$ is strictly monotone. TFAE
    \begin{enumerate}
        \item $\en$ is 1-bibasic;\label{1-UQG:monotone norm:i}
        \item $\en$ is 1-uniformly quasi-greedy;\label{1-UQG:monotone norm:ii}
        \item $\en$ are disjoint.\label{1-UQG:monotone norm:iii}
    \end{enumerate}
\end{prop}
\begin{pf}
    We prove that both~\ref{1-UQG:monotone norm:i} and~\ref{1-UQG:monotone norm:ii} imply \ref{1-UQG:monotone norm:iii}. Assume that~\ref{1-UQG:monotone norm:i} or~\ref{1-UQG:monotone norm:ii} holds. For $i,j\in \N, i<j$ and $t\in\interval{-1}{1}$, we have
    \begin{equation*}
        \norm{\abs{e_i}\vee\abs{e_i+t e_j}} = \norm{e_i+t e_j},        
    \end{equation*}
    which, by the strict monotonicity of the norm, implies that $\abs{e_i}\vee\abs{e_i+t e_j} \ngtr \abs{e_i+t e_j}$ for all $t\in\interval{-1}{1}$. Suppose that $\abs{e_i}\wedge\abs{e_j}>0$. In view of the $C(K)$ representation theorem, there exists $\xi\in K$ such that $e_i(\xi)\neq 0\neq e_j(\xi)$. Choosing $t=\theta \frac{-e_i(\xi)}{e_j(\xi)}$ for some $\theta\in \interval[open left]{0}{1}$ such that $t\in \interval{-1}{1}\setminus{\{0\}}$, we obtain that 
    \begin{equation*}
        \abs{e_i+t e_j}(\xi)<\abs{e_i}(\xi),
    \end{equation*}
    which implies that $\abs{e_i}\vee\abs{e_i+t e_j} > \abs{e_i+t e_j}$, a contradiction. 
\end{pf}
\begin{question}
    If $X$ is not an AM-space and the norm on $X$ is not strictly monotone, does 1-uniformly-quasi-greediness still imply 1-absoluteness? 
\end{question}
\subsection{Absolutely quasi-greedy bases}
Let $\en$ be a uniformly quasi-greedy basis of $E\subset X$. From Thm.~\ref{Charac:UQG}~\ref{Charac:UQG:iv}, we know that the greedy sums of any element $x\in E$ associated to any greedy ordering $\pi$ is uniformly order bounded. However, as shown in~\cite[Ex. 3.10]{UQGBases}, the order bound is, in general, not uniform over all greedy orderings of $x$. This leads us to the following definition of a stronger notion of quasi-greediness.
\begin{definition}
    A basis $\en$ of a Banach subspace $E$ of a Banach lattice $X$ is said to be \textbf{absolutely quasi-greedy} if, for all $x\in E$, there exists $u\in X$ such that
    \begin{equation*}
        \forall \text{ greedy ordering } \pi \text{ of } x, \; \forall m\in \N:\;
        G_{\pi, m}^{\vee}(x)\leq u.
    \end{equation*}
\end{definition}
For uniformly quasi-greedy bases, the following characterization can be derived for absolutely quasi-greediness, which was first proven in~\cite[Prop.~4.6]{BibasicUQGThesis}.
\begin{prop}\label{UniformGO:Equiv:1}
  Let $\en$ be a uniformly quasi-greedy basis. TFAE
  \begin{enumerate}
    \item $\en$ is absolutely quasi-greedy;\label{UniformGO:Equiv:1:i}
    \item For all $x\in E$, there exists $u\in X$ such that
    \begin{equation*}
      \forall\mu>0:\; \sum_{\abs{e^*_n(x)}=\mu} \abs{e^*_n(x)e_n} \leq u.
    \end{equation*}\label{UniformGO:Equiv:1:ii}
  \end{enumerate}
\end{prop}
\begin{proof}
  For~\ref{UniformGO:Equiv:1:i}$\Rightarrow$\ref{UniformGO:Equiv:1:ii}, let $x\in X$ and $u\in X$ such that $u$ dominates the greedy sums of $x$ associated to any greedy ordering. For $\mu>0$ such that $\Set{n\in\N}{\abs{e^*_n(x)}=\mu}\neq \emptyset$, $\sup_{B\subset \Set{n\in\N}{\abs{e^*_n(x)}=\mu}}\abs{\sum_{n\in B}e^*_n(x)e_n}\leq 2u$ since all the sums in the supremum are differences of greedy sums. In view of the $C(K)$ representation theorem, the following relation holds for all $x_1, \ldots, x_m\in\R$ and hence also for arbitrary elements in $X$ 
  \begin{equation*}
    \sum_{k=1}^m \abs{x_k} \leq 2 \sup_{B\subset\{1,\ldots, m\}}\abs{\sum_{k\in B}x_k},
  \end{equation*}
  which, in combination with the above argument, gives that $\sum_{n\in B}\abs{e^*_n(x)e_n}\leq 4u$ and this bound is uniform over $\mu$.\par
  For the converse, let $x\in X$ and $\sigma$ be a fixed greedy ordering of $x$. Any greedy sum $G_{\pi,m}(x)$ for a greedy ordering $\pi$ and $m\in\N$ can be written as
  \begin{equation*}
    G_{\pi,m}(x) = \sum_{\abs{e^*_n(x)}>\mu}e^*_n(x)e_n + \sum_{n\in B}e^*_n(x)e_n
  \end{equation*}  
  for some $\mu>0$ and $B\subset \Set{n\in\N}{\abs{e^*_n(x)}=\mu}$. Since $\sum_{\abs{e^*_n(x)}>\mu}e^*_n(x)e_n$ is a greedy sum of $x$ associated to $\sigma$ and $\en$ is uniformly quasi-greedy, there exists $v\in X$ independent of $\mu$ such that $\abs{\sum_{\abs{e^*_n(x)}>\mu}e^*_n(x)e_n}\leq v$. Hence, using $(ii)$ we obtain
  \begin{equation*}
    \abs{G_{\pi,m}} \leq \abs{\sum_{\abs{e^*_n(x)}>\mu}e^*_n(x)e_n} + \abs{\sum_{B}e^*_n(x)e_n} \leq v + \sum_{B}\abs{e^*_n(x)e_n} \leq v+u
  \end{equation*}
  for some $u\in X$ independent of $\mu$ and $B$ and thus independent of $\pi$ and $m$.
\end{proof}
Using the above characterization, one can prove that the (normalized) Haar system $\chin$ in $L^p(\interval{0}{1})$, $1<p<\infty$, is absolutely quasi-greedy, which was first shown in~\cite[Prop. 4.7]{BibasicUQGThesis}. Let us first recall an important lemma~\cite[Ch. 3, Lemma 1]{OrthSeries} for the Haar system.
\begin{lemma}\label{Haar:Property}
  For any partial sum of the form
  \begin{equation*}
    \sum_{n=M}^{N} a_n\chi_n, \quad 1<M<N
  \end{equation*}
  there exists a permutation $(\sigma(n))_{n=M}^N$ of the numbers $M,M+1,\ldots,N$ such that 
  \begin{equation*}
    \bigvee_{M\leq p\leq q\leq N} \abs{\sum_{n=p}^{q} a_{\sigma(n)}\chi_{\sigma(n)}}
    \geq \frac{1}{4}\sum_{M}^{N} \abs{a_n\chi_n}.
  \end{equation*}
\end{lemma}
\begin{remark}
  Lemma~\ref{Haar:Property} also holds for $\sum_{n\in B} a_n\chi_n$ for any finite subset $B\subset\N\setminus\{1\}$ ordered by the natural order on $\N$. To see this, apply Lemma~\ref{Haar:Property} to $M=\min_{n\in B} n$, $N=\max_{n\in B} n$ and set $a_n=0$ for $n\notin B, M\leq n\leq N$.
\end{remark}
\begin{prop}\label{UniformGO:Haar}
  The Haar system $\chin$ is absolutely quasi-greedy in $L^p(\interval{0}{1})$, $1<p<\infty$. 
\end{prop}
\begin{proof}
  We establish $(ii)$ in Prop.~\ref{UniformGO:Equiv:1}. Let $x\in E$ and WLOG assume that $\chi_1^*(x)=0$. Let $(\mu_l)_{l\in\N}$ be the enumeration of the numbers $(\abs{\chi_n^*(x)})_{n\in\N}$ in descending order with identical numbers counted once. For all $\mu_l$, let $\sigma_l$ be the permutation given by applying Lemma~\ref{Haar:Property} to $\sum_{\abs{e_n^*(x)}=\mu_l} \chi^*_n(x)\chi_n$ and define a greedy ordering $\pi$ of $x$ by
  \begin{equation*}
    \forall \mu_l: \forall i,j\in \Set{n\in\N}{\abs{\chi_n^*(x)}=\mu_l}:\;
    \sigma_l^{-1}(i)<\sigma_l^{-1}(j)\implies\pi^{-1}(i)<\pi^{-1}(j).
  \end{equation*}
  Let $\mu_l$ be arbitrary and $\{k_1, k_2,\ldots, k_m\}$ be the enumeration of $\Set{n\in\N}{\abs{\chi_n^*(x)}=\mu_l}$ in the natural order given by $\N$. For all $1\leq p\leq q\leq m$, there exists $N_1,N_2\in\N, N_1<N_2$ such that
  \begin{equation*}
    \sum_{i=p}^{q} \chi^*_{\sigma_l(k_i)}(x)\chi_{\sigma_l(k_i)} 
    = G_{\pi,N_2}(x)-G_{\pi,N_1}(x)
  \end{equation*}
  by our definition of $\pi$. Using that $\chin$ is uniformly quasi-greedy, there exists a $u\in X$ such that $\abs{G_{\pi,N}(x)}\leq u$ for all $N\in\N$. Hence
  \begin{equation*}
    \abs{\sum_{i=p}^{q} \chi^*_{\sigma_l(k_i)}(x)\chi_{\sigma_l(k_i)}}
    \leq \abs{G_{\pi,N_2}(x)} + \abs{G_{\pi,N_1}(x)}
    \leq 2u
  \end{equation*}
  and this bound is independent of $p,q,\mu_l$. We conclude by taking supremum over $1\leq p\leq q\leq m$ and using Lemma~\ref{Haar:Property}
  \begin{equation*}
    \frac{1}{4}\sum_{i=1}^{m} \abs{\chi^*_{k_i}(x)\chi_{k_i}}
    \leq \bigvee_{1\leq p\leq q\leq m} \abs{\sum_{i=p}^{q} \chi^*_{\sigma_l(k_i)}(x)\chi_{\sigma_l(k_i)}}
    \leq 2u.\qedhere
  \end{equation*}
\end{proof}
Guided by the characterization of uniformly quasi-greedy bases in Thm.~\ref{Charac:UQG}, we now present a maximal-inequality type characterization of absolutely quasi-greediness. To that end, let us denote the set of greedy orderings of an $x\in E$ as $\Gamma(x)$ and note that for a given $m\in \N$, the set of $m$-th greedy sums of $x$ is finite, which implies that the set of greedy sums of $x$ is countable. Hence, by considering the lexicographic order on $(m,\pi)\in \N\times\Gamma(x)$, we view $\Set{G_{\pi,m}(x)}{\pi\in \Gamma(x),\; m\in\N}$ as a sequence in the following and introduce the notation
\begin{equation*}
    G_m^{\vee\vee}(x)\coloneqq \bigvee_{\pi\in\Gamma(x)}G_{\pi,m}^{\vee}(x). 
\end{equation*}
\begin{thm}\label{UniformGO:Charac}
    Let $\en$ be a semi-normalized basis. TFAE
    \begin{enumerate}
        \item For all $x\in E$, $G_{\pi,m}(x)\xrightarrow{u}x$;\label{UniformGO:Charac:i}
        \item For all $x\in E$, $G_{\pi,m}(x)\xrightarrow{\sigma o}x$;\label{UniformGO:Charac:ii}
        \item For all $x\in E$, $G_{\pi,m}(x)\xrightarrow{o}x$;\label{UniformGO:Charac:iii}
        \item For all $x\in E$, there exists $u\in X$ such that $G_{\pi,m}^{\vee}(x)\leq u$ for all greedy orderings $\pi$ of $x$ and $m\in \N$;\label{UniformGO:Charac:iv}
        \item For all $x\in E$, there exists $C>0$ such that $\norm{G_m^{\vee\vee}(x)}\leq C$ for all greedy orderings $\pi$ of $x$ and $m\in \N$;\label{UniformGO:Charac:v}
        \item There exists $C\geq 1$ such that
        \begin{equation*}\label{UniformGO:max ineq}
            \norm{G_m^{\vee\vee}(x)}\leq C\norm{x}       
        \end{equation*}
        for all $m\in\N$ and $x\in E$ with finite support.    \label{UniformGO:Charac:vi}
    \end{enumerate}
\end{thm}
\begin{proof}
    \ref{UniformGO:Charac:i}$\Rightarrow$ \ref{UniformGO:Charac:ii}$\Rightarrow$ \ref{UniformGO:Charac:iii}$\Rightarrow$\ref{UniformGO:Charac:iv}$\Rightarrow$ \ref{UniformGO:Charac:v} is clear.\par
    We prove~\ref{UniformGO:Charac:v}$\Rightarrow$\ref{UniformGO:Charac:vi} by contradiction. Assume~\ref{UniformGO:Charac:vi} does not hold. Then we claim that $(e_n)_{n\geq 2}$ also fails~\ref{UniformGO:Charac:vi}. Suppose not. Then there exists $M\geq 1$ such that for all $x\in E$ with finite support in $\N\setminus\{1\}$
    \begin{equation*}
        \norm{G_{\abs{\supp(x)}}^{\vee\vee}(x)}\leq M\norm{x}.
    \end{equation*}
    Now consider $x\in E$ with finite support and $e^*_1(x)\neq 0$. Define $y\coloneqq x - e^*_1(x)e_1$. For any $\pi_x\in\Gamma(x)$, we can find $\pi_y\in\Gamma(y)$ such that 
    \begin{equation*}
        G_{\pi_x, \abs{\supp(x)}}^{\vee}(x) \leq \abs{e^*_1(x)e_1} + G_{\pi_y, \abs{\supp(x)-1}}^{\vee}(y). 
    \end{equation*}
    Taking supremum over $\pi_y\in\Gamma(y)$, $\pi_x\in\Gamma(x)$, and then taking norm on both sides gives
    \begin{align*}
        \norm{G_{\abs{\supp(x)}}^{\vee\vee}(x)} 
        &\leq \norm{e^*_1(x)e_1} + \norm{G_{ \abs{\supp(x)-1}}^{\vee\vee}(y)}\\
        &\leq \norm{e^*_1(x)e_1} + M\norm{y} = \norm{e^*_1(x)e_1} + M\norm{x-e^*_1(x)e_1}\\
        &\leq (1+M)\norm{e^*_1(x)e_1} + M\norm{x} \leq ((1+M)K_b + M)\norm{x},
    \end{align*}
    which contradicts the assumption and hence proves the claim. Proceeding inductively, we obtain that $(e_n)_{n\geq N}$ fails~\ref{UniformGO:Charac:vi} for every $N\in\N$. Now we show that there exists $(x_k)\subset E$ such that 
    \begin{enumerate}
        \item $(\supp(x_k))_{k\in\N}$ are finite and disjoint;
        \item $\norm{x_k}\leq 2^{-k}$;
        \item $\norm{G_{\supp(x_k)}^{\vee\vee}(x_k)}\geq k$;
        \item $\max_{n\in \supp(x_{k+1})}\abs{e^*_n(x_k+1)}<\min_{n\in \supp(x_k)}\abs{e^*_n(x_k)}$,
    \end{enumerate}
    which will contradict~\ref{UniformGO:Charac:v}. Indeed, given such a sequence $(x_k)$, we can define $x\coloneqq \sum_{k=1}^{\infty} x_k$. Since $G_{\pi,n}(x_k)$ is the difference of greedy sums of $x$ for all $k\in\N$, $\pi\in\Gamma(x_k)$ and $1\leq n\leq \abs{\supp(x_k)}$,~\ref{UniformGO:Charac:v} would guarantee uniform norm boundedness of $G_{\supp(x_k)}^{\vee\vee}(x_k)$, which contradicts property (iii) in the above construction.\par  
    We proceed with an inductive construction of $(x_k)$. The existence of $x_1$ is given by our assumption. Suppose that we have constructed $x_1, \ldots, x_{k-1}$.
    Let $\mu\coloneqq \min_{n\in \supp(x_{k-1})}\abs{e^*_n(x_{k-1})}$ and $N\coloneqq \max_{n\in \supp(x_{k-1})} n$. By the above argument, we can find an $x_k\in [(e_n)_{n\geq N+1}]$ with finite support such that 
    \begin{equation*}
        \norm{G_{\abs{\supp(x_k)}}^{\vee\vee}(x_k)}\geq C_k\norm{x_k},
    \end{equation*}
    where $C_k\coloneqq\max \{k2^{k}, k\cdot4K_b (\inf_{n\in\N}\norm{e_n})^{-1}\mu^{-1}\}$. Scaling $x_k$ we can take $\norm{x_k}=kC_k^{-1}$, so that 
    \begin{gather*}
        \norm{x_k}\leq 2^{-k}, \quad \max_{n\in \supp(x_k)}\abs{e_n^*(x_k)}\leq 2K_b(\inf_{n\in\N}\norm{e_n})^{-1}\norm{x_k}\leq \mu/2,\\
        \norm{G_{\abs{\supp(x_k)}}^{\vee\vee}(x_k)}\geq k.
    \end{gather*}
    For~\ref{UniformGO:Charac:vi}$\Rightarrow$\ref{UniformGO:Charac:i}, we first note that~\ref{UniformGO:Charac:vi} implies that $\en$ is uniformly quasi-greedy by Thm.~\ref{Charac:UQG}~\ref{Charac:UQG:vi}. Hence, for a given $x\in E$, we can find strictly greedy sums $(G_{\pi,m_k}(x))_{k\in\N}$ such that $G_{\pi,m_k}(x)\xrightarrow[k\to\infty]{u}x$. Passing to a further subsequence, we may assume that
    \begin{equation*}
        \norm{G_{\pi,m_k+1}(x) - G_{\pi,m_k}(x)}\leq 2^{-k}.
    \end{equation*}
    Define $y_k\coloneqq G_{\pi,m_k+1}(x) - G_{\pi,m_k}(x)$ and $u_k\coloneqq G_{\abs{supp(y_k)}}^{\vee\vee}(y_k)$. By~\ref{UniformGO:Charac:vi}, $\norm{u_k}\leq C\norm{y_k}\leq C2^{-k}$, hence we may define $e\coloneqq \sum_{k=1}^{\infty}ku_k$. It follows that $u_k\leq \frac{e}{k}$, which implies that $u_k\xrightarrow{u}0$.\par
    Since $\abs{G_{\pi,m_k}(x) - x} + u_k\xrightarrow{u}0$, there exists $f\in X$ such that for any $\epsilon>0$ there exists $N_{\epsilon}$ such that 
    \begin{equation*}
        \forall k\geq N_{\epsilon}:\; \abs{G_{\pi,m_k}(x) - x} + u_k\leq \epsilon f. 
    \end{equation*}
    Fix $\epsilon$ and find the required $N_{\epsilon}$. For $m\geq m_{N_{\epsilon}}$, there exists $k\geq N_{\epsilon}$ such that $m_k\leq m<m_{k+1}$. For any $\pi\in\Gamma(x)$ we have then
    \begin{equation*}
        \abs{G_{\pi, m}(x) - x} 
        \leq \abs{G_{\pi, m_k}(x) - x} + \abs{G_{\pi, m}(x) - G_{\pi, m_k}(x)}
        \leq \abs{G_{\pi, m_k}(x) - x} + u_k\leq \epsilon f,
    \end{equation*}
    which shows that $G_{\pi,m}(x)\xrightarrow{u}x$.
\end{proof}
\begin{remark}
    By considering initial segments of $x\in E$ and passing to the limit, it can be shown that Thm.~\ref{UniformGO:Charac}\ref{UniformGO:Charac:vi} also holds for $x\in E$ with infinite support.
\end{remark}
\begin{remark}\label{UniformGO:Charac:Rem:1}
    Since Thm.~\ref{UniformGO:Charac}~\ref{UniformGO:Charac:vi} does not depend on the ambient space, we can replace $X$ by $X^{**}$ in~\ref{UniformGO:Charac:i}-\ref{UniformGO:Charac:vi}.
\end{remark}
The least constant $C$ satisfying Thm.~\ref{UniformGO:Charac}~\ref{UniformGO:Charac:vi} will be referred to as the \textbf{absolutely quasi-greedy constant} of $\en$ and denoted as $K_{aqg}$. Parallel to~\cite[Prop. 3.20]{UQGBases}, the following proposition shows that absolutely quasi-greedy bases enjoy the property of being absolute for constant coefficients.
\begin{prop}\label{UniformGO:Property:1}
  Let $\en$ be an absolutely quasi-greedy basis. Then there exists a constant $C\geq 1$ such that
  \begin{equation*}
    \forall (n_i)_{i=1}^{k}\subset \N:\, \forall (\epsilon_i)_{i=1}^k\subset \{\pm 1\}^k:\; \norm{\sum_{i=1}^k \abs{e_{n_i}}} \leq C \norm{\sum_{i=1}^k \epsilon_i e_{n_i}}. 
  \end{equation*}
\end{prop}
\begin{proof}
  Since $\en$ is quasi-greedy, $\norm{\sum_{i=1}^k \epsilon_i e_{n_i}}$ and $\norm{\sum_{i=1}^k e_{n_i}}$ are equivalent up to a multiplicative constant of $2K_{qg}$ by~\cite[Prop. 10.2.10]{TopicsBanach}. Hence, it suffices to show the existence of a constant $C\geq 1$ such that 
  \begin{equation}\label{Ineq:3}
      \forall (n_i)_{i=1}^{k}\subset \N:\; \norm{\sum_{i=1}^k \abs{e_{n_i}}} \leq C \norm{\sum_{i=1}^k e_{n_i}}. 
  \end{equation}
  We claim that if $(e_n)_{n>m}$ satisfies~\eqref{Ineq:3} for some $m\in\N$, namely
  \begin{equation*}
    \forall (n_i)_{i=1}^{k}\subset \N\setminus\{1,\ldots, m\}: \norm{\sum_{i=1}^k \abs{e_{n_i}}} \leq C\norm{\sum_{i=1}^k e_{n_i}}
  \end{equation*}
  for some $C\geq 1$, then $\en$ satisfies~\eqref{Ineq:3} as well. Indeed, let $(n_i)_{i=1}^{k}$ be an arbitrary finite subset of $\N$. We have
  \begin{equation*}
    \norm{\sum_{i=1}^k \abs{e_{n_i}}}
    \leq \norm{\sum_{n_i\leq m} \abs{e_{n_i}}} + \norm{\sum_{n_i>m} \abs{e_{n_i}}} 
    \leq \sum_{n_i\leq m}\norm{e_{n_i}} + \norm{\sum_{n_i>m} \abs{e_{n_i}}}.
  \end{equation*}
  Since $\en$ is quasi-greedy, $\norm{e_{n_i}}\leq K_{qg} \norm{\sum_{i=1}^k e_{n_i}}$ and hence
  \begin{equation*}
    \norm{\sum_{i=1}^k \abs{e_{n_i}}}
    \leq mK_{qg} \norm{\sum_{i=1}^k e_{n_i}} + \norm{\sum_{n_i>m} \abs{e_{n_i}}}
    \leq (mK_{qg}+C) \norm{\sum_{i=1}^k e_{n_i}}.
  \end{equation*}
  Now assume that $\en$ does not satisfy~\eqref{Ineq:3}. Then there exists $(n_i^1)_{i=1}^{k_1}\subset \N$ such that $\norm{\sum_{i=1}^{k_1} \abs{e_{n_i^1}}} \geq 2 \norm{\sum_{i=1}^{k_1} e_{n_i^1}}$. Let $m_1 \coloneqq \max_{i=1}^{k_1} n_i^1$. By the above argument, $(e_n)_{n>m_1}$ does not satisfy~\eqref{Ineq:3}, so there exists $(n_i^2)_{i=1}^{k_2}\subset \N\setminus\{1,\ldots,m_1\}$ such that $\norm{\sum_{i=1}^{k_2} \abs{e_{n_i^2}}} \geq 4 \norm{\sum_{i=1}^{k_2} e_{n_i^2}}$. Proceeding iteratively, we find a disjoint family $(B_l)_{l\in\N}$ of finite subsets of $\N$ such that
  \begin{equation*}
    \norm{\sum_{n\in B_l} \abs{e_n}} \geq 2^l \norm{\sum_{n\in B_l} e_n}.
  \end{equation*}
  Let us define $(a_l)_{l\in\N}\subset\R_{>0}$ such that $a_l\norm{\sum_{n\in B_l} e_n} = 2^{-l/2}$. By the quasi-greediness of $\en$, $\norm{\sum_{n\in C} a_l e_n} \leq K_{qg}\norm{\sum_{n\in B_l}a_l e_n} = K_{qg}2^{-l/2}$ for all subsets $C$ of $B_l$, which implies that the series $\sum_{l=1}^{\infty} \sum_{n\in B_l} a_l e_n$ converges to some $x\in E$. By assumption and Prop.~\ref{UniformGO:Equiv:1}, there exists $u\in X$ such that
  \begin{equation}\label{Ineq:4}
    \forall \mu>0:\; \sum_{\abs{e^*_n(x)}=\mu} \abs{e^*_n(x)e_n} \leq u.
  \end{equation}
  However, choosing $\mu=a_l$ we obtain that
  \begin{equation*}
    \norm{\sum_{n\in B_l} \abs{a_l e_n}} = a_l \norm{\sum_{n\in B_l} \abs{e_n}}
    \geq 2^l a_l\norm{\sum_{n\in B_l} e_n} = 2^{l/2}\xrightarrow{l\to\infty}\infty,
  \end{equation*}
  which contradicts~\eqref{Ineq:4}.
\end{proof}
\begin{remark}
    Prop.~\ref{UniformGO:Property:1} implies that the Rademacher sequence is not absolutely quasi-greedy, which is a simplified version of the counterexample in~\cite[Example 3.10]{UQGBases}. In particular, we see that permutability does not imply absolutely quasi-greediness. On the other hand, the Haar system in $L^p(\interval{0}{1})$, $1<p<\infty$, is absolutely quasi-greedy but not permutable~\cite[Example 6.2]{BibasicSeq}. 
\end{remark}
We proceed with some results on the stability under scaling of absolutely quasi-greediness, the development of which resembles the case of uniformly quasi-greediness.
\begin{prop}\label{Scaling:Equiv:2}
  Let $\en$ be an absolutely quasi-greedy basis and $a>1$. TFAE
  \begin{enumerate}
    \item For all $\lambdan$ with $1\leq\abs{\lambda_n}\leq a$, $\lambdaen$ is absolutely quasi-greedy;\label{Scaling:Equiv:2:i}
    \item For all $x\in E$, there exists $u\in X$ such that 
    \begin{equation*}
      \forall k\in\Z:\;\sum_{A^a_k(x)}\abs{e^*_n(x)e_n}\leq u;
    \end{equation*}\label{Scaling:Equiv:2:ii}
    \item For all $x\in E$, there exists $C>0$ such that 
    \begin{equation*}
      \forall N\in\N:\;\norm{\bigvee_{k=-N}^{k=N}\sum_{A^a_k(x)}\abs{e^*_n(x)e_n}}\leq C.
    \end{equation*}\label{Scaling:Equiv:2:iii}
  \end{enumerate}
\end{prop}
\begin{proof}
  For~\ref{Scaling:Equiv:2:i}$\Rightarrow$~\ref{Scaling:Equiv:2:ii}, assume that~\ref{Scaling:Equiv:1:i} holds and let $x\in E$. Using the same idea as in the proof of Prop.~\ref{Scaling:Equiv:1}, we can find a scaling $\lambdan$ with $1\leq \lambda_n \leq a$ such that $\frac{\abs{e^*_n(x)}}{\lambda_n}$ is constant on $A_k^a(x)$ for all $k\in \Z$. Since $\lambdaen$ is absolutely quasi-greedy by assumption,~\ref{Charac:UQG:ii} follows from Prop.~\ref{UniformGO:Equiv:1}~\ref{UniformGO:Equiv:1:ii}. \par 
  For~\ref{Scaling:Equiv:2:ii}$\Rightarrow$~\ref{Scaling:Equiv:2:i}, we first observe that~\ref{Scaling:Equiv:2:ii} is stronger than Prop.~\ref{Scaling:Equiv:1}~\ref{Scaling:Equiv:1:ii} and thus implies that $\lambdaen$ is uniformly quasi-greedy for all $\lambdan$ with $1\leq \abs{\lambda_n} \leq a$. Fix such a scaling $\lambdan$ and let $\mu>0$ be arbitrary. Let $k\in\Z$ such that $\mu\in\interval[open left]{a^k}{a^{k+1}}$. It holds that
  \begin{align*}
      \Set{n\in\N}{\frac{\abs{e_n^*(x)}}{\abs{\lambda_n}}=\mu}
      &\subset \Set{n\in\N}{\frac{\abs{e_n^*(x)}}{\abs{\lambda_n}}\in \interval[open left]{a^k}{a^{k+1}}}\\
      &\subset \Set{n\in\N}{\abs{e_n^*(x)}\in \interval[open left]{a^k}{a^{k+2}}} = A^a_{k}(x)\sqcup A^a_{k+1}(x),
  \end{align*}
  which implies that
  \begin{align*}
      \sum_{\frac{\abs{e^*_n(x)}}{\abs{\lambda_n}}=\mu}
      \abs{e^*_n(x)e_n}
      \leq \sum_{A^a_{k}(x)\sqcup A^a_{k+1}(x)}
      \abs{e^*_n(x)e_n}
      &= \sum_{A^a_{k}(x)}
      \abs{e^*_n(x)e_n} + 
      \sum_{A^a_{k+1}(x)}
      \abs{e^*_n(x)e_n}
      \leq 2u
  \end{align*}
  by the assumption of~\ref{Scaling:Equiv:2:ii}. By Prop.~\ref{UniformGO:Equiv:1} we deduce that $\lambdaen$ is absolutely quasi-greedy.\par
  \ref{Scaling:Equiv:2:ii}$\Rightarrow$\ref{Scaling:Equiv:2:iii} is clear. It remains to prove that \ref{Scaling:Equiv:2:iii}$\Rightarrow$\ref{Scaling:Equiv:2:i}. Since $\bigvee_{k=-N}^{k=N}\sum_{A^a_k(x)}\abs{e^*_n(x)e_n}$ increases with $N$,~\ref{Scaling:Equiv:2:iii} implies that $\sum_{A^a_k(x)}\abs{e^*_n(x)e_n}$ is uniformly order bounded over $k\in\Z$ in $X^{**}$. Since~\ref{Scaling:Equiv:2:i}$\Leftrightarrow$\ref{Scaling:Equiv:2:ii} in $X^{**}$, we obtain that $\lambdaen$ is absolutely quasi-greedy for all $\lambdan$ with $1\leq \abs{\lambda_n} \leq a$ if we consider $E$ to be a subspace of $X^{**}$. We conclude by Remark~\ref{UniformGO:Charac:Rem:1}.
\end{proof}
\begin{remark}
    In the proof of~\ref{Scaling:Equiv:2:i}$\Rightarrow$~\ref{Scaling:Equiv:2:ii}, we see that if $\en$ is absolutely quasi-greedy for all scalings $\lambdan$ with $1\leq \abs{\lambda_n}\leq a$, then the greedy sums of $x\in E$ are uniformly order bounded over all greedy orderings with respect to all such scalings.
\end{remark}
\begin{prop}\label{Scaling:Uncond:Uniform:1}
  Let $\en$ be an unconditional absolutely quasi-greedy basis. Then $\lambdaen$ is absolutely quasi-greedy for all $\lambdan$ with $0<\inf_{n\in\mathbb{N}}\abs{\lambda_n}\leq\sup_{n\in\mathbb{N}}\abs{\lambda_n}<\infty$.
\end{prop}
\begin{proof}
  Let $a>1$. We establish Prop.~\ref{Scaling:Equiv:2}~\ref{Scaling:Equiv:2:ii} by exploiting unconditionality. Using the same idea as in the proof of Prop.~\ref{Scaling:Equiv:1}, we can find a scaling $\lambdan$ with $1\leq \lambda_n \leq a$ such that $\frac{\abs{e^*_n(x)}}{\lambda_n}=\min_{n\in A_k^a(x)}(\abs{e^*_n(x)})$ on $A_k^a(x)$ for all $k\in \Z$. By the unconditionality of $\en$, we can define $\bar{x}\coloneqq \sum_{n=1}^{\infty} \frac{e^*_n(x)}{\lambda_n}e_n$. Applying Prop.~\ref{UniformGO:Equiv:1} and taking $\mu=\min_{n\in A_k^a(x)}(\abs{e^*_n(x)})$, there is a $u\in X$ independent of $k\in\Z$ such that
  \begin{equation*}
    \sum_{A_k^a(x)} \abs{\frac{e^*_n(x)}{\lambda_n}e_n}\leq u.
  \end{equation*}
  Prop.~\ref{Scaling:Equiv:2}~\ref{Scaling:Equiv:2:ii} then follows from the boundedness of $\lambdan$.
\end{proof}
\begin{ex}
    Since the Haar system in $L^p(\interval{0}{1})$, $1<p<\infty$, is unconditional and absolutely quasi-greedy, Prop.~\ref{Scaling:Uncond:Uniform:1} asserts that the scaled Haar system is still absolutely quasi-greedy for all scalings with bounded quotient. 
\end{ex}
\begin{prop}\label{Scaling:UniformGO:Property:1}
  If $\lambdaen$ is absolutely quasi-greedy for all $\lambdan$ with $1\leq\abs{\lambda_n}\leq a$, then the supremum of the absolutely quasi-greedy constants over all such $\lambdaen$ is finite. 
\end{prop}
\begin{proof}
    We prove by contradiction. Assume that the supremum is infinite. We claim then that the supremum of the absolutely quasi-greedy constants over all such $(\lambda_n e_n)_{n\geq 2}$ is also infinite. The rest of the proof follows the same idea as in the proof of Prop.~\ref{Scaling:Property:1}. To prove the claim, let us assume that there exists a constant $M\geq 1$ such that for all $x\in E$ with finite support and $e^*_1(x)=0$ and for all scalings $\lambdan$ with $1\leq \abs{\lambda_n}\leq a$, we have
    \begin{equation*}  
        \norm{\tilde{G}_{\abs{\supp(x)}}^{\vee\vee}(x)}\leq M\norm{x},
    \end{equation*}
    where we used tilde to indicate that the greedy orderings are with respect to the scaled basis. For any $x\in E$ with finite support and $e^*_1(x)\neq 0$, a scaling $\lambdan$ with $1\leq\abs{\lambda_n}\leq a$, and a greedy ordering $\pi_{\lambda}$ of $x$ with respect to $\lambdaen$, we have then
    \begin{equation*}
    \tilde{G}_{\pi_{\lambda},\abs{\supp(x)}}^{\vee}(x) 
    \leq \abs{e^*_1(x)e_1} + \tilde{G}_{\pi_{\lambda},\abs{\supp(x)}-1}^{\vee}(x-e^*_1(x)e_1).
  \end{equation*} 
  Taking supremum over greedy orderings and then taking norm on both sides yields
  \begin{align*}
    \norm{\tilde{G}_{\abs{\supp(x)}}^{\vee\vee}(x)} 
    &\leq \norm{e^*_1(x)e_1} + \norm{\tilde{G}_{\abs{\supp(x)}-1}^{\vee\vee}(x-e^*_1(x)e_1)}\\
    &\leq \norm{e^*_1(x)e_1} + M\norm{x-e^*_1(x)e_1}\\
    &\leq (1+M)\norm{e^*_1(x)e_1} + M\norm{x}\\
    &\leq ((1+M)K_b+M)\norm{x}.
  \end{align*}
  This contradicts the assumption and therefore proves the claim.
\end{proof}
\section*{Acknowledgments}
The author would like to thank Mitchell Taylor for numerous valuable discussions.



\bibliographystyle{plain}
\bibliography{references.bib}

\end{document}